\documentclass[a4paper,11pt, twoside]{amsart}

\usepackage{amsmath,amsfonts,amsthm,amsopn,color,amssymb,enumitem,mathrsfs,cite}
\usepackage{palatino}
\usepackage[colorinlistoftodos]{todonotes}
\usepackage[colorlinks=true,urlcolor=blue, citecolor=red,linkcolor=blue,linktocpage,pdfpagelabels, bookmarksnumbered,bookmarksopen]{hyperref}
\usepackage[hyperpageref]{backref}

\usepackage{cancel}

\usepackage[left=2.61cm,right=2.61cm,top=2.72cm,bottom=2.72cm]{geometry}

\numberwithin{equation}{section}

\newcommand{\C}{\mathbb{C}}
\newcommand{\R}{\mathbb{R}}

\newcommand{\RD}{{\mathbb{R}^2}}

\newcommand{\de}{\partial}

\renewcommand{\le}{\leslant}
\renewcommand{\ge}{\geslant}
\renewcommand{\a }{\alpha }

\newcommand{\g }{\gamma }

\renewcommand{\l }{\lambda}

\newcommand{\n }{\nabla }
\newcommand{\s }{\sigma }

\newcommand{\X}{\mathcal{X}}

\renewcommand{\H}{H^1(\RD)}
\newcommand{\Hr}{H^1_r(\RD)}

\newcommand{\M}{\mathcal{M}}

\newcommand{\N}{\mathbb{N}}

\renewcommand{\C}{\mathbb{C}}
\renewcommand{\o}{\omega}

\def\bbm[#1]{\mbox{\boldmath $#1$}}
\newcommand{\beq }{\begin{equation}}
\newcommand{\eeq }{\end{equation}}

\renewcommand{\le}{\leqslant}
\renewcommand{\ge}{\geqslant}

\newcommand{\ird}{\int_{\RD}}

\providecommand{\pgfsyspdfmark}[3]{}

\makeatletter
\providecommand\@dotsep{5}
\def\listtodoname{List of Todos}
\def\listoftodos{\@starttoc{tdo}\listtodoname}
\makeatother

\newtheorem{theorem}{Theorem}[section]
\newtheorem{lemma}[theorem]{Lemma}
\newtheorem{remark}[theorem]{Remark}
\newtheorem{proposition}[theorem]{Proposition}

\newcommand{\BA}{\mathbf{A}}
\newcommand{\ef}{\eqref}

\title
[Modified Schr\"odinger-Chern-Simons equations]
{Standing waves of modified Schr\"odinger equations
	coupled with the Chern-Simons gauge theory}

\author[P. d'Avenia]{Pietro d'Avenia}
\author[A. Pomponio]{Alessio Pomponio}
\author[T. Watanabe]{Tatsuya Watanabe}

\address[P. d'Avenia]{\newline\indent
Dipartimento di Meccanica, Matematica e Management
\newline\indent 
Politecnico di Bari
\newline\indent
Via Orabona 4,  70125  Bari, Italy}
\email{pietro.davenia@poliba.it}

\address[A. Pomponio]{\newline\indent
Dipartimento di Meccanica, Matematica e Management
\newline\indent 
Politecnico di Bari
\newline\indent
Via Orabona 4,  70125  Bari, Italy}
\email{alessio.pomponio@poliba.it}

\address[T. Watanabe]{\newline\indent 
Department of Mathematics, 
\newline\indent 
Faculty of Science, Kyoto Sangyo University,
\newline\indent
Motoyama, Kamigamo, Kita-ku, Kyoto-City, 603-8555, Japan}
\email{tatsuw@cc.kyoto-su.ac.jp}

\thanks{}

\subjclass[2010]{35J62, 35J20, 35Q60}
\date{}
\keywords{ground state solution, gauged Schr\"odinger equation, variational method}

\begin{document}

\begin{abstract}
We are interested in standing waves of a 
modified Schr\"odinger equation coupled with the Chern-Simons gauge theory.
By applying a constraint minimization of Nehari-Pohozaev type, 
we prove the existence of radial ground state solutions.
We also investigate the non-existence  for nontrivial solutions.

\end{abstract}

\maketitle

\section{Introduction}

In this paper, we consider the following nonlocal quasilinear elliptic problem
\begin{equation}\label{eq:1.1}
\begin{split}
&-\Delta u +\omega u -\mu u \Delta u^2
+q \frac{h_{u}^{2}(|x|)}{|x|^{2}} (1 +\mu u^2)u \\
&\qquad + q \Big( \int_{|x|}^\infty \frac{h_u(s)}{s} \big(2+ \mu u^2(s) \big) u^2(s) \,ds 
\Big)u =
\lambda |u|^{p-1}u \quad \hbox{in} \quad \R^2, 
\end{split}
\end{equation}
where $u:\R^2 \to \R$ is radially symmetric, 
$\omega$, $\mu$, $q$, $\lambda$ are positive constants, 
$p>1$ and 
\[
h_u(s)= \int_0^s ru^2(r) \,dr, \quad s \ge 0.\]
Equation \ef{eq:1.1} appears in the study of standing waves for
a modified Schr\"odinger equation coupled with the Chern-Simons gauge theory.
For reader's convenience, we will give the derivation of \ef{eq:1.1} in Section \ref{se:der}.

Our aim of this paper is to study existence and non-existence of positive radial solutions
and radial ground state solutions of \ef{eq:1.1}.

When $q=0$, \ef{eq:1.1} is reduced to the following quasilinear elliptic problem
\begin{equation} \label{eq:1.2}
-\Delta u +\omega u -\mu u \Delta u^2
=\lambda |u|^{p-1}u,
\end{equation}
which is obtained by the {\it modified Schr\"odinger equation}
\begin{equation} \label{eq:1.3}
{\rm i} \psi_{t}+\Delta \psi+\mu \psi \Delta |\psi|^2 +\lambda |\psi|^{p-1}\psi=0,
\end{equation}
looking for standing waves of the form $\psi(t,x)=\exp(i \omega t)u(x)$.
In the last decades, a considerable attention has been devoted to the study of solutions 
to the quasi-linear Schr\"odinger equation \eqref{eq:1.3} that  arises in various fields of Physics (see \cite{BEPZ, BHZ, BL, KBRW, K}). This model is known to be more accurate in many physical phenomena compared with
the classical semi-linear Schr\"odinger equation ${\rm i} \psi_{t}+\Delta \psi+|\psi|^{p-1}\psi=0$.
In \eqref{eq:1.3} $\psi:\R \times \R^2 \to \C$ and 
$\l$ is a constant representing the strength of self-interaction potential.
Moreover, the {\em additional term} $\mu\psi \Delta |\psi|^2 $ appears in various physical models
and arises due to:
\begin{itemize}
	\item the nonlocality of the nonlinear interaction for electron (see \cite{BEPZ}),
	\item the weak nonlocal limit for nonlocal nonlinear Kerr media \cite{KBRW},
	\item the surface term for superfluid film (see \cite{K}),
\end{itemize}
and the parameter $\mu$ represents the strength of each effect and may not be small.

The existence and properties of ground states of \ef{eq:1.2}
as well as stability of standing wave solutions 
have been also studied widely,
see e.g. \cite{AW, CJ, CJS, LWW, RS, SW} and references therein. 

On the other hand if $\mu=0$ in \ef{eq:1.1}, 
one obtains the following non-local elliptic problem
\begin{equation} \label{eq:1.4}
-\Delta u +\omega u 
+q \frac{h_{u}^{2}(|x|)}{|x|^{2}} u 
+ 2q \Big( \int_{|x|}^\infty \frac{h_u(s)}{s} u^2(s) \,ds 
\Big)u =
\lambda |u|^{p-1}u. 
\end{equation}
Equation \ef{eq:1.4} appears in the study of nonlinear Schr\"odinger equations
coupled with the Chern-Simons gauge fields.
Recently, a lot of works concerning with \ef{eq:1.4} has been done, 
see \cite{byeon, BHS2, cunha, huh2, huh3, JPR, P,  AD, AD2, WT, Y}.
Here we briefly introduce some known results on \ef{eq:1.4}. 
In \cite{byeon}, the existence of a positive radial solution of \ef{eq:1.4}
was shown in the case $p>3$ by using a suitable constraint minimization argument.
The authors in \cite{byeon} also investigated the case $1<p \le 3$.
They obtained existence and non-existence results depending on $\lambda$ for the case $p=3$,
and the existence of positive radial solutions as minimizers under $L^2$-constraint
in the case $1<p<3$  ($\o$ appears as a Lagrangian multiplier).
When $p>3$, the existence of a positive solution in the non-radial setting 
has been also obtained in \cite{WT}.
In \cite{AD}, a detailed study for the case $1<p<3$ has been performed.
The authors in \cite{AD} investigated the geometry of the functional 
associated with \ef{eq:1.4} and obtained an explicit threshold value for $\omega$.
They also showed the multiple existence of positive radial solutions 
for $\omega$ in some range.
We mention that in \cite{AD2} the case of a bounded domain for $1<p<3$ is considered
and some results on boundary concentration of solutions has been proved.
In \cite{cunha} the authors studied \ef{eq:1.4} with general nonlinearities of 
the Berestycki-Lions type, 
and obtained a multiplicity result when $q$ is sufficiently small.
We also recall that the multiple existence of normalized solutions of \ef{eq:1.4}
has been studied in \cite{Y}.
Finally, we refer to \cite{berge, liu, tataru} for results on 
Cauchy problem associated with \ef{eq:1.4}.
To summarize, the existence and the non-existence of solutions of \ef{eq:1.4}
heavily depends on $\omega$, $q$, $\lambda$ and $p$,
and the solution set of \ef{eq:1.4} has a rich structure 
depending on the parameters and the exponent $p$.

The purpose of this paper is to investigate the structure 
of the solutions set for \ef{eq:1.1}, which seems to be more complicated
due to the presence of the quasilinear term. 

\medskip
To state our main result, let us define the metric space
\[
\X:=\{u\in H_r^1(\RD) : u^2\in \H\},
\]
endowed with the distance
\[
d_{\X}(u,v):=\|u-v\|_{H^1}+\|\n (u^2)-\n (v^2) \|_{2}.
\]
We recall that 
\[
\Hr:=\{u\in H^1(\RD) : u \hbox{ is radially symmetric}\}.
\] 
Then $u \in \X$ is called a {\em weak solution} of \ef{eq:1.1} if $u$ satisfies
\begin{equation}\label{eq:1.5}
\begin{split}
&\int_{\R^2} \Bigg\{ (1+2 \mu u^2) \nabla u \cdot \nabla \varphi 
+2 \mu u | \nabla u|^2 \varphi
+ \omega u \varphi- \lambda |u|^{p-1}u \varphi 
+q \frac{h_u^2(|x|)}{|x|^2} (1+ \mu u^2) u\varphi \\
&\qquad 
+q \left( \int_{|x|}^{\infty} \frac{h_u(s)}{s} \big( 2+ \mu u^2(s) \big) u^2(s)\,ds
\right) u \varphi \Bigg\} \,dx =0 ,\quad 
\hbox{for all } \varphi \in C_{0,r}^{\infty}(\R^2),
\end{split}
\end{equation}
where $C_{0,r}^{\infty}(\R^2):=\{u\in C_0^{\infty}(\R^2) : u \hbox{ is radially symmetric}\}$.

At least formally, 
weak solutions of \ef{eq:1.1} can be obtained 
as critical points of the following functional defined on $\X$
\begin{equation} \label{funct}
\begin{split}
I(u)&=  \frac 12 \ird \left[ (1+2\mu u^2)|\n u|^2+\omega u^2 \right] \,dx
+ \frac{q}{2} \ird \frac{u^2(x)}{|x|^2}\left( \int_{0}^{|x|}su^2(s) \,ds\right)^2 \,dx 
\\
&\qquad 
+\frac{q\mu}{4}\ird \frac{u^4(x)}{|x|^2}\left( \int_{0}^{|x|}su^2(s) \,ds \right)^2 \,dx
-\frac \lambda {p+1}\ird |u|^{p+1} \,dx,
\end{split}
\end{equation}
but $\X$ is not a vector space, because it is not close with respect to the sum. So we cannot speak about {\em critical points} of $I$ in the usual way, since the functional is not differentiable. However, as we will see in Section \ref{se:vs}, one can show that $I$ is well-defined and continuous on $\X$.
Moreover, since, for every given $u \in \X$ and $\varphi \in  C_{0,r}^{\infty}(\R^2)$, we have $u+\varphi\in \X$, we can evaluate  
the Gateaux derivative
\[
\begin{split}
I'(u)[\varphi] 
&=\int_{\R^2} \Bigg\{ (1+2 \mu u^2) \nabla u \cdot \nabla \varphi
+2 \mu u | \nabla u|^2 \varphi
+ \omega u \varphi- \lambda |u|^{p-1}u \varphi  \\
&\quad +q \frac{h_u^2(|x|)}{|x|^2} (1+ \mu u^2) u\varphi
+q \left( \int_{|x|}^{\infty} \frac{h_u(s)}{s} \big( 2+ \mu u^2(s) \big) u^2(s)\,ds
\right) u \varphi \Bigg\} \,dx.
\end{split}
\]
Then $u \in \X$ is a weak solution of \ef{eq:1.1} if and only if 
the Gateaux derivative of $I$ in every direction $\varphi \in  C_{0,r}^{\infty}(\R^2)$ is zero (see Lemma \ref{lem:3.3} below).

Our main results are the followings.

\medskip
\begin{theorem}\label{main}
Assume that $p>5$. 
Then for any $\omega$, $\mu$, $q$ and $\lambda>0$, 
\eqref{eq:1.1} has a positive radial solution $u \in \X \cap C^2(\R^2)$.
Moreover $u$ is a radial ground state of \ef{eq:1.1}, that is, 
$u$ has least energy among any nontrivial radial weak solutions of \ef{eq:1.1}.
\end{theorem}

\begin{theorem} \label{thm2}
Assume that $1<p<5$. Then, for any $\mu$, $q$ and $\lambda>0$, 
there exists $\bar{\omega}>0$ such that for $\omega \ge \bar{\omega}$, 
\ef{eq:1.1} has no nontrivial solution.
\end{theorem}
We will also study the dependence of $\bar{\omega}$ with respect to $\mu$ 
and $q$ in Remark \ref{rem:5.3} below.

\medskip
To prove Theorem \ref{main}, we use a constraint minimization argument
which is a combination of the Nehari manifold and the Pohozaev manifold,
as performed in \cite{byeon, WT}.
However we must pay attention to apply this approach in our case,
since the functional $I$ associated with \ef{eq:1.1} is only Gateaux differentiable and only in some directions.
We will overcome this difficulty by establishing the regularity 
of weak solutions of \ef{eq:1.1}.
Once we could show that any weak solution of \ef{eq:1.1} satisfies the Nehari
identity and the Pohozaev identity, 
we next aim to prove that the constraint minimizer is actually 
a ground state solution. 
For this purpose, we apply an argument performed in \cite{LWW, RS},
which enables us to avoid considering complicated algebraic equations
as in \cite{byeon, WT}.

The proof of Theorem \ref{thm2} can be done similarly as in \cite{byeon, AD}.
To this end, we will obtain a new inequality of Sobolev type for $u \in \X$.
As shown in \cite{AD} for the case $\mu=0$, 
the existence and the non-existence of positive solutions of \ef{eq:1.1}
in the case $1<p<5$ heavily depends on $\omega$, $\mu$, $q$ and $\lambda$.
We expect to obtain the (multiple) existence of positive solutions 
when $\omega$ is small. 
But we postpone this question to a future work.

\medskip
This paper is organized as follows. 
In Section \ref{se:der}, we introduce the derivation of \ef{eq:1.1}
and the role of physical constants $\omega$, $\mu$, $q$ and $\lambda$.
We formulate \ef{eq:1.1} as a variational problem in Section \ref{se:vs}.
The regularity property of weak solutions of \ef{eq:1.1},
which enables us to apply the Pohozaev identity and plays a central role
for the existence of ground state solutions, 
 is also established here.
In Section \ref{se:proof1}, we will obtain the existence result (Theorem \ref{main})
by applying the constraint minimization technique described before.
Finally, by establishing a new inequality of Sobolev type for $u \in \X$,
we prove Theorem \ref{thm2} in Section \ref{se:proof2}.

\section{Derivation of the model}\label{se:der}
In this section, we introduce the derivation of Equation \ef{eq:1.1}
together with physical backgrounds.
Let us consider the Lagrangian density $\mathcal{L}_{\rm MNLS}$
for a modified nonlinear Schr\"odinger equation, which is given by
\begin{equation}
\label{LagMNLS}
\mathcal{L}_{\rm MNLS}
=\frac{1}{2} \Im(\psi \bar{\psi_t})
-\frac{1}{2}| \nabla \psi|^2 + \frac{\l}{p+1} |\psi|^{p+1}
-\frac{\mu}{4} \big| \nabla |\psi|^2 \big|^2.
\end{equation}

We are interested in the situation where the Schr\"odinger wave function $\psi$ is, for instance, a charged particle and
interacts with the gauge potential $(\phi, \BA)$ 
for the electro-magnetic field in the Chern-Simons theory.
Here $\phi:\R \times \R^2 \to \R$
and $\BA=\big( A^1,A^2 \big): \R \times \R^2 \to \R^2$ 
are the electric potential and the magnetic potential respectively.
Then 
the interaction between $\psi$ and $(\phi,\BA)$ is 
obtained by replacing the usual derivatives with the covariant ones, namely
\begin{equation} \label{cov}
\partial_t \longmapsto  \partial_t +ie\phi,
\qquad
\nabla \longmapsto \nabla - ie{\bf A},
\end{equation}
where $e$ denotes the strength of the interaction with the electro-magnetic field
(see \cite{Fel} for details).
Substituting \ef{cov} in \ef{LagMNLS}, one has the following Lagrangian
\[
\tilde{\mathcal{L}}_{\rm MNLS}(\psi,\phi,{\bf A})
=\frac{1}{2} \Im (\psi \bar{\psi_t})-\frac{e}{2} \phi |\psi|^2
-\frac{1}{2} |\nabla\psi- ie{\bf A}\psi|^2
+\frac{\lambda}{p+1}|\psi|^{p+1}
- \frac{\mu}{4} \big| \nabla|\psi|^2- ie{\bf A}|\psi|^2 \big|^2.
\]
We have to consider also the Lagrangian density for the electro-magnetic field, which, in the Chern-Simons theory,
is given by
\[
\mathcal{L}_{\rm MCS}(\phi,\BA)
=-\frac{1}{8} F^{\alpha \beta} F_{\alpha \beta} + 
\frac{\kappa}{8} \varepsilon^{\nu \alpha \beta} A_{\nu} F_{\alpha \beta},
\quad F_{\alpha \beta}= \partial_{\alpha} A_{\beta} - \partial_{\beta} A_{\alpha}, \ 
\alpha,\beta,\nu\in \{ 0,1,2 \},
\]
where the first term in $\mathcal{L}_{\rm MCS}$ is the usual Maxwell term and
the second term is the so-called Chern-Simons term (see \cite{jackiw0, jackiw2} for details).
Here $\varepsilon$ is the Levi-Civita tensor,
$\kappa \in \R$ is a parameter which controls the Chern-Simons term,
the Lorentz metric tensor is ${\rm diag}(1,-1,-1)$, and
the coordinates are 
$x^{\nu}=(t,x_1,x_2)$. Moreover we have $A^0=A_0=\phi$ and $A^j=-A_j$, for $j=1,2$.
At large distances and low energies, the lower derivatives of the Chern-Simons term
dominates the higher derivative Maxwell term, and hence we may replace the Lagrangian 
density by
\[
	\tilde{\mathcal{L}}_{\rm MCS}(\phi,\BA)
	=\frac{\kappa}{8} \varepsilon^{\nu \alpha \beta} A_{\nu} F_{\alpha \beta}
	=\frac{\kappa}{4} \big(
	\phi(\partial_2 A^1 - \partial_1 A^2) - A^1(\partial_2 \phi + \partial_t A^2) 
	+ A^2(\partial_t A^1 + \partial_1 \phi)
	\big). \]
So, the total Lagrangian $\mathcal{L}$ is defined by
\begin{align*}
	\mathcal{L}(\psi,\phi,\BA)
	&=
	\tilde{\mathcal{L}}_{\rm MNLS}(\psi,\phi,\BA) + \tilde{\mathcal{L}}_{\rm MCS}(\phi,\BA) \\
	&=
	\frac{1}{2} \Im (\psi \bar{\psi_t})-\frac{e}{2} \phi |\psi|^2
	-\frac{1}{2} |\nabla\psi- ie{\bf A}\psi|^2
	+\frac{\lambda}{p+1}|\psi|^{p+1}
	- \frac{\mu}{4} \big| \nabla|\psi|^2- ie{\bf A}|\psi|^2 \big|^2 \\
	&\quad +\frac{\kappa}{4} \big(
	\phi(\partial_2 A^1 - \partial_1 A^2) - A^1(\partial_2 \phi + \partial_t A^2) 
	+ A^2(\partial_t A^1 + \partial_1 \phi)
	\big).
	\end{align*}
Then the Euler-Lagrange equations for the total action
\[
\mathcal{S}=\mathcal{S}(\psi,\phi,\BA)=\int_{\R^2} \int_{\R} \mathcal{L}(\psi,\phi,\BA) \,dt \,dx
\]
are given by
\begin{equation*}\label{eulero}
	\begin{cases}
	i \psi_t-e\phi \psi +( \nabla -ie \BA)^2 \psi 
	+ \l |\psi|^{p-1} \psi + \mu \psi \Delta |\psi|^2 -e^2 \mu |\BA|^2 |\psi|^2 \psi=0
	\\
	\kappa (\partial_2 A^1-\partial_1 A^2) =e|\psi|^2
	\\
	\kappa( \partial_2 \phi+\partial_t A^2)
	+e^2\mu|\psi|^4 A^1 
	=2e \Im \Big( \bar{\psi} (\partial_1 \psi-ie A^1 \psi) \Big)
	\\
	-\kappa( \partial_1 \phi+\partial_t A^1)
	+e^2\mu|\psi|^4 A^2 
	=2e \Im \Big( \bar{\psi} (\partial_2 \psi-ie A^2 \psi) \Big).
	\end{cases}
	\end{equation*}

If we consider standing waves  
$\psi(t,x)={\rm exp}\big( iS(t,x) \big) u(t,x)$ with $u$, $S:\R \times \R^2 \to \R$, the total action depends on $(u,S,\phi,\BA)$ and the Euler-Lagrange equations become
\begin{equation*}
	\begin{cases}
	-\Delta u +(S_t +e \phi +| \nabla S-e \BA|^2)u
	-\mu u \Delta u^2 +e^2 \mu |\BA|^2 u^3=\l |u|^{p-1} u
	\\
	\partial_t (u^2) +2 \operatorname{div}\big( (\nabla S-e \BA) u^2 \big)=0
	\\
	\kappa (\partial_2 A^1 - \partial_1 A^2) =eu^2
	\\
	\kappa (\partial_2 \phi + \partial_t A^2) 
	+ e^2 \mu u^4 A^1
	=2e (\partial_1 S - eA^1)  u^2
	\\
	-\kappa (\partial_1 \phi +\partial_t A^1)+ e^2 \mu u^4 A^2
	=2e (\partial_2 S - eA^2) u^2.
	\end{cases}
	\end{equation*}
Now we suppose that $u=u(x)$ and $S=\omega t$.
Moreover we consider the static case: 
$\phi=\phi(x)$ and $A^i=A^i(x)$. 
Then we get
\[
	\begin{cases}
	-\Delta u +\omega u -\mu u \Delta u^2 +e\phi u
	+e^2|{\bf A}|^2 (1 +\mu u^2)u
	= \lambda |u|^{p-1}u, \\
	\operatorname{div}({\bf A}u^2) =0, \\
	\kappa (\partial_2 A^1 - \partial_1 A^2) =eu^2, \\
	-\kappa \partial_2 \phi 
	= e^2 (2+ \mu u^2)u^2 A^1, \\
	\kappa  \partial_1 \phi
	= e^2 (2+ \mu u^2)u^2 A^2.
	\end{cases}
	\]
Finally in the Coulomb gauge $\operatorname{div}\BA=0$, 
it follows that $\operatorname{div}({\bf A}u^2)= {\bf A}\cdot \n u^2$ and hence
\begin{equation}
	\label{lastsyst}
	\begin{cases}
	-\Delta u +\omega u -\mu u \Delta u^2 +e\phi u
	+e^2|{\bf A}|^2 (1 +\mu u^2)u
	= \lambda |u|^{p-1}u, \\
	{\bf A}\cdot \n u^2=0, \\
	\kappa (\partial_2 A^1 - \partial_1 A^2) =eu^2, \\
	-\kappa \partial_2 \phi 
	= e^2(2+\mu  u^2)  u^2 A^1, \\
	\kappa  \partial_1 \phi
	= e^2(2+\mu  u^2) u^2 A^2.
	\end{cases}
	\end{equation}

Observe that the second equation in \ef{lastsyst} implies that, up to the ``trivial cases'', 
the function $u$ is radial  if and only if $\mathbf A$ is a {\em tangential} vector field, i.e. 
$$
\mathbf A= \frac{e}{\kappa} h_{u}(x)\mathbf t,
\quad
\hbox{where } \quad 
\mathbf t=(x_{2}/|x|^{2},-x_{1}/|x|^{2}).$$ 
Moreover, since the problem is invariant by translations, to avoid the related difficulties, 
we look for radial solutions $u$. 
Thus, from this choice, arguing as in \cite[Lemma 3.3]{huh2}, it follows that
$\mathbf A$ has to be invariant for the group action:
$$
{\rm T}_{g}\mathbf A(x)= g^{-1}\cdot \mathbf A (g(x)), 
\quad g\in O(2),$$
and this readily implies that $h_{u}$ has to be a radial function.
So, whenever $u$ is radial, the magnetic potential $\BA$ has to be necessarily written as
\begin{equation*}
	A^1(x)=\frac{e}{\kappa}\frac{x_2}{|x|^{2}}h_{u}(|x|),
	\quad
	A^2(x)=-\frac{e}{\kappa}\frac{x_1}{|x|^{2}}  h_{u}(|x|).
	\end{equation*}
Moreover, by the last two equations in system \eqref{lastsyst}, one finds that
\begin{equation*} 
\n \phi 
= \frac{e^2}{\kappa} \left( A^2 , -A^1 \right) (2 + \mu u^2)u^2 
=
-\frac{e^3}{\kappa^2} \frac{x}{|x|^{2}}h_{u}(|x|) (2 + \mu  u^2)u^2
=
-\frac{e^3}{\kappa^2} h_{u}(|x|) (2 + \mu  u^2)u^2 \mathbf n
\end{equation*}
where $\mathbf n=(x^{1}/|x|^{2}, x^{2}/|x|^{2})$. 
Thus it follows that the electric potential $\phi$ is radial. Assuming that 
$\displaystyle \lim_{|x|\to+\infty}\phi(|x|)=0$, we have
\[
\phi(|x|)=\frac{e^3}{\kappa^2} \int_{|x|}^\infty 
\frac{h_u(s)}{s} \big( 2 + \mu u^2(s) \big) u^2(s) \,ds.
\]
Finally, using the third equation in system \eqref{lastsyst} and assuming $h_{u}(0)=0$, 
which is necessary to have $\mathbf A$ smooth, we have
\[
h_{u}(|x|)=\int_{0}^{|x|} s u^{2}(s) \,ds.
\]
In this way we have solved $\phi$ and $\mathbf A$ in terms of $u$ and so, in order to solve the  \eqref{lastsyst}, we need to study only the first equation of the system 
which, now,  can be written as
\begin{multline*} 
 -\Delta u +\omega u -\mu u \Delta u^2
+\frac{e^4}{\kappa^2}\frac{h_{u}^{2}(|x|)}{|x|^{2}} (1 +\mu u^2)u \\
\quad + \frac{e^4}{\kappa^2} 
\Big( \int_{|x|}^\infty \frac{h_u(s)}{s} \big( 2 + \mu u^2(s) \big) u^2(s) \,ds \Big) u
= \lambda |u|^{p-1}u \quad \hbox{in }  \R^2. 
\end{multline*}
Putting $q:= \dfrac{e^4}{\kappa^2}$, we arrive at \ef{eq:1.1}.

\section{Variational settings and preliminaries}\label{se:vs}

In this section, we formulate \ef{eq:1.1} as a variational problem 
and prepare some preliminary results.
Now we observe that if $u\in \X$ is a solution of \ef{eq:1.1}, then it solves $L(u)=0$ where
\[
L(u)= \operatorname{div}A(u, \nabla u) +B(x,u,\nabla u),
\] 
with
\begin{align} \label{eq:3.1}
A(\s, {\bf p})& = (1+2 \mu \s^2) {\bf p}, \nonumber \\
B(x,\s,{\bf p}) &=
-\big( 2 \mu |{\bf p}|^2 + \omega +qV_1(x)(1+\mu u^2) +q V_2(x) \big) \s
+\lambda |\s|^{p-1}\s,
\end{align}
and 
\[
V_1(x)= \begin{cases}
\displaystyle \frac{h_u^2(|x|)}{|x|^2} & x\neq 0,\\
0 & x=0,
\end{cases} \qquad
V_2(x)= \int_{|x|}^{\infty} \frac{h_u(s)}{s} \big( 2+ \mu u^2(s) \big) u^2(s) \,ds.\]
From \ef{eq:3.1}, we find that \ef{eq:1.1} is a quasilinear elliptic equation
with principal part in divergence form 
and the structure conditions in \cite{LU} are all fulfilled 
(see \cite[Chapter 4]{LU} or \cite{Tol}).

First we establish that any weak solutions of \ef{eq:1.1} are classical ones.
To this end, we begin with the following lemma.

\begin{lemma} \label{lem:3.1}
Let us fix $u \in \X$.  We have:
\begin{enumerate}[label=(\roman*),ref=\roman*]
	\item \label{lem31i}$V_1$, $V_2$ are non-negative and bounded;
	\item \label{lem31ii}if we suppose further that $u \in C(\R^2)$, then $V_1$ and $V_2$ belong to the class $C^1(\R^2)$.
\end{enumerate}
\end{lemma} 

\begin{proof}
We argue as in \cite[Proposition 2.1, 2.2]{byeon}.
First by the definition, we see that $V_1$, $V_2$ are non-negative.
Next by the Schwarz inequality, one finds that
\begin{equation} \label{eq:3.2}
h_u(s) = \frac{1}{2 \pi} \int_{B_s(0)} u^2(y) \,dy
\le \frac{1}{2\pi} |B_s(0)|^{\frac{1}{2}} \| u\|_{4}^2
\le Cs \| u \|_{4}^2, \quad \hbox{for} \ s \ge 0.
\end{equation} 
Thus by the definition and from \ef{eq:3.2}, we get
\begin{equation*} 
V_1(x) \le C \| u\|_{4}^4, \quad \hbox{for all} \ x \in \R^2.
\end{equation*}
Moreover, observing that, for all $ x\in \R^2$, $0\le V_2(x)\le V_2(0)$, we need to estimate only
\begin{equation*}
\begin{split}
V_2(0) &= 
\int_{0}^1 \frac{h_u(s)}{s} \big( 2+ \mu u^2(s) \big) u^2(s) \,ds
+\int_1^{\infty} \frac{h_u(s)}{s} \big( 2+ \mu u^2(s) \big) u^2(s) \,ds  \\
&\le C \| u \|_{4}^2 
\left[\int_{0}^1 (2+\mu u^2(s))u^2(s)\,ds
+\int_1^{\infty} (2+ \mu u^2(s))u^2(s) \,ds\right] \\
&\le C \| u\|_{4}^2 \left[\left( \int_{0}^1 s^{-\frac{1}{2}} \,ds \right)^{\frac{2}{3}}
\left( \int_{0}^1 (2+\mu u^2(s))^3 u^6(s) s \,ds \right)^{\frac{1}{3}}
+ \int_1^{\infty} (2+ \mu u^2(s))u^2(s) s\,ds\right] \\
&\le C\| u\|_{4}^2(\| u\|_{6}^2+\| u\|_{{12}}^4
+\| u\|_{2}^2+\| u\|_{4}^4).
\end{split}
\end{equation*}
This completes the proof of (\ref{lem31i}).\\
To prove (\ref{lem31ii}), we observe that $V_1$, $V_2 \in C^1(\R^2 \setminus \{ 0\})$ 
if $u \in C(\R^2)$. Moreover since $u \in C(\R^2)$, it follows that
\[
\frac{h_u(|x|)}{|x|^2} 
= \frac{1}{2 \pi |x|^2} \int_{B_{|x|}(0)} u^2(y) \,dy
\to \frac{1}{2} u^2(0), \quad \hbox{as} \ |x| \to 0.\]
This implies that $h_u(|x|)=O(|x|^2)$ as $|x| \to 0$.
Thus one has
$V_1(x)
=O(|x|^2)$ and, for $i=1,2$, 
$$
\frac{\partial V_1}{\partial x_i}(0)=
0,$$
and
\[
\frac{\partial V_1}{\partial x_i} (x)
= \frac{2 x_i h_u(|x|)}{|x|^4} \big( |x|^2 u^2(x) -h_u(|x|) \big)
=O(|x|), \quad \hbox{as} \ |x| \to 0,\]
from which we conclude that $V_1 \in C^1(\R^2)$.
In a similar way, it follows that $V_2 \in C^1(\R^2)$.
\end{proof}

Now we are ready to prove the following regularity result.

\begin{proposition} \label{prop:3.2}
Let $u \in \X$ be a weak solution of \ef{eq:1.1}. 
Then $u \in C^2(\R^2)$ and decays exponentially up to second derivatives.
\end{proposition}

\begin{proof}
The proof consists of two steps.\\
{\bf Step 1}: We claim that $u \in L^{\infty}(\R^2)$ and $u(x) \to 0$, as $|x| \to \infty$.\\
For this purpose, we perform the De Giorgi iteration as in 
\cite[Theorem 7.1]{LU}, \cite[Appendix 6]{LWW}.
Let $y \in \R^2$, $R>0$ and $\sigma \in (0,1)$ be arbitrarily given.
Choose a cut-off function $\xi \in C_0^{\infty}(\R^2)$ with
$\xi =1$ on $B_{\sigma R}(y)$, 
$\xi =0$ on $B_R^c(y)$, $0 \le \xi \le 1$ and
$|\nabla \xi | \le \frac{C}{(1-\sigma)R}$.
Finally we set $\varphi= \xi^2(u-k)_+$ with $k \ge 0$.\\
By a density argument,  one can use $\varphi$ as a test function in \ef{eq:1.5}
to obtain
\begin{align*}
&\int_{\R^2} \xi^2(1+2\mu u^2)|\nabla (u-k)_+|^2
+ \xi^2 \big( 2 \mu |\nabla u|^2 +\omega +qV_1(x)(1+\mu u^2) +qV_2(x) \big)
u(u-k)_+ \,dx \\
&= -2 \int_{\R^2} \xi (1+2 \mu u^2) (u-k)_+ \nabla u \cdot \nabla \xi \,dx
+\lambda \int_{\R^2} \xi^2 |u|^{p-1} u(u-k)_+ \,dx.
\end{align*}
Since $V_1$, $V_2 \ge 0$, $(u-k)_+=0$ on $\{ u \le k \}$ 
and $0 \le (u-k)_+ \le u$ on $\{ u>k \}$, one has
\begin{align*}
&\int_{\R^2} \xi^2 \big( 2 \mu |\nabla u|^2 +\omega +qV_1(x)(1+\mu u^2) +qV_2(x) \big)
u(u-k)_+ \,dx \\
&\ge 2 \mu \int_{ \{ u>k \} } \xi^2 u (u-k)_+ |\nabla u|^2 \,dx \\
&\ge 2\mu \int_{ \{ u>k \} } \xi^2 (u-k)_+^2 |\nabla u|^2 \,dx \\
&= 2 \mu \int_{\R^2} \xi^2 (u-k)_+^2 |\nabla u|^2 \,dx.
\end{align*}
On the other hand by $\nabla (u-k)_+=\nabla u$ on the set $\{u>k\}$, 
the H\"older inequality and the Young inequality,
we also have
\begin{align*}
&-2 \int_{\R^2} \xi (1+2 \mu u^2) (u-k)_+ \nabla u \cdot \nabla \xi \,dx \\
&\le 2 \int_{\{ u>k \}} \xi (1+2\mu u^2) (u-k)_+ | \nabla u| |\nabla \xi| \,dx \\
&\le 2 \left( \int_{\{ u>k \}} \xi^2 (1+2\mu u^2) |\nabla u|^2 \,dx \right)^{\frac{1}{2}}
\left( \int_{\{ u>k \}} (1+2\mu u^2)(u-k)_+^2 |\nabla \xi|^2 \,dx \right)^{\frac{1}{2}} \\
&\le \frac{1}{2} \int_{\{ u>k \}} \xi^2(1+2\mu u^2)|\nabla u |^2dx
+2 \int_{\{ u>k \}} (1+2\mu u^2)(u-k)_+^2 |\nabla \xi|^2 \,dx \\
&= \frac{1}{2} \int_{\R^2} \xi^2(1+2\mu u^2)|\nabla (u-k)_+|^2dx
+2 \int_{\{ u>k \}} (1+2\mu u^2)(u-k)_+^2 |\nabla \xi|^2 \,dx.
\end{align*}
Thus it follows that
\begin{equation}\label{eq:3.5}
\begin{split}
& \frac{1}{2} \int_{\R^2} \xi^2 (1+2 \mu u^2)|\nabla (u-k)_+|^2 \,dx
+2 \mu \int_{\R^2} \xi^2 (u-k)_+^2 |\nabla u|^2 \,dx  \\
&\le 2 \int_{ \{ u>k \}} (1+2 \mu u^2)(u-k)_+^2 |\nabla \xi|^2 \,dx
+ \lambda \int_{ \{ u>k \}} \xi^2 |u|^{p+1} \,dx.
\end{split}
\end{equation}
Now we put $v:= \sqrt{1+2\mu u^2}(u-k)_++k$ so that
$(1+2 \mu u^2)(u-k)_+^2 =(v-k)_+^2$. 
Then it follows that
\begin{equation} \label{eq:3.6}
u>k \ \Leftrightarrow \ v>k \quad
\hbox{and} \quad u \le v \ \ \hbox{a.e.} \ x\in \R^2.
\end{equation}
Moreover by the Young inequality, one has
\begin{align*}
| \nabla (v-k)_+|^2
&= (1+2 \mu u^2)| \nabla (u-k)_+|^2 
+4 \mu u (u-k)_+ \nabla u \cdot \nabla (u-k)_+
+ \frac{4\mu^2 u^2}{1+2 \mu u^2} (u-k)_+^2 |\nabla u|^2 \\
&\le C \Big( (1+2\mu u^2)|\nabla (u-k)_+|^2 
+(u-k)_+^2 |\nabla u |^2 \Big).
\end{align*}
Thus from \ef{eq:3.5} and \ef{eq:3.6}, we get
\begin{equation} \label{eq:3.7}
\int_{\R^2} \xi^2 |\nabla (v-k)_+|^2 \,dx
\le C \left\{\int_{ \{ v>k \} } (v-k)_+^2 |\nabla \xi|^2 \,dx
+ \int_{ \{ v>k \}} \xi^2 |u|^{p+1} \,dx \right\}.
\end{equation} 
Next by the H\"older inequality, one has
\begin{align*}
\int_{ \{ v>k \}} \xi^2 |u|^{p+1} \,dx
&\le \left( \int_{ \{ v>k \} } \xi^2 \,dx \right)^{\frac{1}{2}}
\left( \int_{ \{ v>k \}} \xi^2 |u|^{2p+2} \,dx \right)^{\frac{1}{2}} \\
&\le \left( \int_{ B_R(y)} |u|^{2p+2} \,dx \right)^{\frac{1}{2}} |A_{k,R}^+|^{\frac{1}{2}},
\end{align*}
where $A_{k,R}^+:= \{ x\in B_R(y) : \ v(x)>k \}$.
From \ef{eq:3.7}, we find that
\[
\int_{B_{\sigma R}(y)} | \nabla (v-k)_+|^2 \,dx
\le C \left\{
\frac{1}{(1-\sigma)^2R^2} \int_{B_R(y)} (v-k)_+^2 \,dx
+\| u\|_{L^{2p+2}(B_R(y))}^{p+1} | A_{k,R}^+|^{\frac{1}{2}} \right\}, \]
for any $\sigma \in (0,1)$ and $k \ge 0$.
This implies that $v$ belongs to the De Giorgi class $DG^+$ and hence,
by 
\cite{DT}, we have
\begin{equation*} 
\sup_{ x \in B_{\sigma R}(y)} v_+(x)
\le C \left( \| v\|_{L^2(B_R(y))} + \| u\|_{L^{2p+2}(B_R(y))}^{(p+1)/2} \right),
\end{equation*}
and so $v_+\in L^\infty(\RD)$. Since $u \le v$, we deduce that $u_+\in L^\infty (\RD)$, too.
Arguing similarly, one can show that $u_-$ is bounded from above.
This yields that $u \in L^{\infty}(\R^2)$. 
Finally since $u\in \Hr$, by the radial lemma due to \cite{strauss}, 
$u$ decays to zero at infinity.\\
{\bf Step 2}: We claim that $u \in C^2(\R^2)$ and
decays exponentially up to second derivatives.\\
By Step 1, we know that $u \in H^1 (\R^2) \cap L^{\infty}(\R^2)$.
Although we only have $V_1$, $V_2 \in L^{\infty}(\R^2)$ at this stage,
we find that $u \in C^{1,\alpha}(\R^2)$ for some $\alpha \in (0,1)$
by applying the regularity result due to \cite{Tol}.
Then by (\ref{lem31ii}) of Lemma \ref{lem:3.1} and the Schauder estimate, 
we conclude that $u \in C^{2,\alpha}(\R^2)$. 
Finally the exponential decay follows by applying suitable
comparison argument
(see e.g. \cite[Theorem 4.1]{RaS}).
This completes the proof.
\end{proof}

Arguing as in \cite{byeon}, standard computations show that
\begin{lemma} \label{lem:3.3}
The functional $I$ in \ef{funct}
is well-defined and continuous in $\X$. Moreover, if the Gateaux derivative of $I$ evaluated in $u\in \X$ is zero in every direction $\varphi \in  C_{0,r}^{\infty}(\R^2)$, then $u$  is a weak solution of \ef{eq:1.1}.
\end{lemma}

We conclude this section with the following
\begin{lemma} \label{lem:3.4}
Any weak solution $u$ of \ef{eq:1.1} satisfies the Nehari identity $N(u)=0$ and
the Pohozaev identity $P(u)=0$, where 
\begin{align}
N(u)&=\int_{\R^2} \left\{ (1+4 \mu u^2) |\nabla u|^2 + \omega u^2 - \lambda |u|^{p+1}
+q \frac{h^2_u(|x|)}{|x|^2} (3+ 2\mu u^2) u^2 \right\} \,dx, \label{eq:3.10} \\
P(u) &= \int_{\R^2} \left\{ \omega u^2 -\frac{2 \lambda}{p+1} |u|^{p+1} 
+q \frac{h^2_u(|x|)}{|x|^2} (2 + \mu u^2) u^2 \right\} \,dx. \label{eq:3.11}
\end{align}
\end{lemma}

\begin{proof}
First by a density argument, one can use $u \in \X$ as a test function in \eqref{eq:1.5}.
Then we see that the identity $N(u)=0$ holds.\\
Next let $u\in \X\cap C^2(\R^2)$ be a solution of \eqref{eq:1.1}. 
Then multiplying by $\n u \cdot x$ and integrating by parts on $B_R$, 
arguing as in \cite[Proposition 2.3]{byeon}, we have
\begin{align*}
&\int_{B_R}\Delta u(\n u \cdot x)\ dx
=\frac R2 \int_{\de B_R}|\n u|^2\ d\s 
=o_R(1),
\\
&\int_{B_R}u\Delta u^2(\n u \cdot x)\ dx
=\frac 12 \int_{B_R}\Delta u^2(\n u^2 \cdot x)\ dx
=\frac R4 \int_{\de B_R}|\n u^2|^2\ d\s 
=o_R(1),
\\
&\int_{B_R} u(\n u \cdot x)\ dx
=-\int_{B_R}u^2\ dx 
+o_R(1),
\\
&\int_{B_R}|u|^{p-1} u(\n u \cdot x)\ dx
=-\frac 2{p+1} \int_{ B_R}|u|^{p+1}\ dx+o_R(1).
\end{align*}
Moreover, for $\a =2$ or $\a =4$, we have
\begin{align*}
&\frac 4\a \int_{B_R}\left(\int_{|x|}^\infty \frac{h_u(s)}s u^\a (s)ds\right)u (\n u \cdot x)\ dx
+\int_{B_R}\frac{h_u^2(|x|)}{|x|^2} u^{\a -1}  (\n u \cdot x)\ dx
\\
&\qquad=\int_{B_R}\frac{h_u^2(|x|)}{|x|^2} u^{\a -1} (\n u \cdot x)\ dx
+\frac 4 \a \int_{B_R}\frac{u^\a(x)}{|x|^2}\left(\int_0^{|x|}su^2(s)\ ds\right)\left(\int_0^{|x|}s^2u(s)u'(s)\ ds\right)dx
\\
&\qquad \qquad-\frac 4 \a \int_{B_R}\frac{u^\a(x)}{|x|^2}\left(\int_0^{|x|}su^2(s)\ ds\right)\left(\int_0^{|x|}s^2u(s)u'(s)\ ds\right)dx
\\
&\qquad \qquad+\frac 4\a \int_{B_R}\left(\int_{|x|}^\infty \frac{h_u(s)}s u^\a (s)ds\right)u (\n u \cdot x)\ dx
\\
&\qquad =\frac 1 \a \left.\frac{d}{dt}\right|_{t=1}\int_{B_R}\frac{u^\a(tx)}{|x|^2}\left(\int_0^{|x|}su^2(ts)\ ds\right)^2 dx
\\
&\qquad \qquad-\frac 4 \a \int_{B_R}\frac{u^\a(x)}{|x|^2}\left(\int_0^{|x|}su^2(s)\ ds\right)\left(\int_0^{|x|}s^2u(s)u'(s)\ ds\right)dx
\\
&\qquad \qquad+\frac 4\a \int_{B_R}\left(\int_{|x|}^\infty \frac{h_u(s)}s u^\a (s)ds\right)u (\n u \cdot x)\ dx
\\
&\qquad =-\frac 4 \a \int_{B_R}\frac{u^\a(x)}{|x|^2}\left(\int_0^{|x|}su^2(s)\ ds\right)^2 dx
+o_R(1).
\end{align*}
The fact that all the remaining terms $o_R(1)$ go to zero as $R\to +\infty$ can be proved as in \cite{byeon} since $u\in \X$ and so $u^2\in \H$.
Thus from \ef{eq:1.1}, one has
\[
\int_{B_R} \left\{ \omega u^2 -\frac{2 \lambda}{p+1} |u|^{p+1} 
+q \frac{h^2_u(|x|)}{|x|^2} (2 + \mu u^2) u^2 \right\} \,dx
+o_R(1)=0,\]
from which we deduce that $P(u)=0$.
\end{proof}

\section{Proof of Theorem \ref{main}}\label{se:proof1}
Throughout this section, we suppose that $p>5$.
In the following, for any $u \in \X$, we denote 
\begin{align*}
A(u)&=\ird |\n u|^2 \,dx,
\qquad
B(u)=\ird u^2 \,dx,
\qquad
C(u)=\ird u^2 |\n u|^2 \,dx,
\\
D(u)&=\ird \frac{u^2(x)}{|x|^2}\left( \int_{0}^{|x|}su^2(s)\ ds\right)^2 \,dx ,
\\
E(u)&=  \ird \frac{u^4(x)}{|x|^2}\left( \int_{0}^{|x|}su^2(s)\ ds\right)^2 \,dx,
\qquad
F(u)=\ird  |u|^{p+1} \,dx. 
\end{align*}

First we recall the following properties of $D(u)$.
\begin{lemma}[{\cite[Lemma 3.2]{byeon}}]\label{le:compD}
Suppose that a sequence $\{u_n\}$ converges weakly to a function $u$ in $\Hr$ as $n \to +\infty$. Then for each $\varphi \in \Hr$, $\{D(u_n)\}$, $\{D'(u_n)[\varphi]\}$ and $\{D'(u_n)[u_n]\}$ converge up to a subsequence to $D(u)$, $D'(u)[\varphi]$ and $D'(u)[u]$, respectively, as $n \to +\infty$.
\end{lemma}

Next we show that analogous properties hold for $E(u)$.
\begin{lemma}\label{le:compE}
Suppose that a sequence $\{u_n\}$ converges weakly to a function $u$ in $\Hr$ as $n \to +\infty$. Then for each $\varphi \in \Hr$, $\{E(u_n)\}$, $\{E'(u_n)[\varphi]\}$ and $\{E'(u_n)[u_n]\}$ converge up to a subsequence to $E(u)$, $E'(u)[\varphi]$ and $E'(u)[u]$, respectively, as $n \to +\infty$.
\end{lemma}

\begin{proof}
First we prove that $E(u_n)\to E(u)$ as $n \to +\infty$.
Now one has
\begin{align*}
(2\pi)^2|E(u_n)-E(u)|
&\le \ird |u_n^4(x)-u^4(x)|  \left(\frac{1}{|x|}\int_{B_{|x|}} u_n^2(y) \,dy \right)^2 \,dx \\
&\quad +\ird u^4(x) \left|\left(\frac{1}{|x|}\int_{B_{|x|}} u_n^2(y) \,dy \right)^2 
- \left(\frac{1}{|x|}\int_{B_{|x|}} u^2(y) \,dy \right)^2\right| \,dx \\
&\le \|u_n^4-u^4\|_2
\left\| \frac1{|\cdot|}\int_{B_{|\cdot|}} u_n^2(y) \,dy \right\|_4 ^2\\
&\quad +\|u\|_8^4
\left\| \left(\frac1{|\cdot|}\int_{B_{|\cdot|}} u_n^2(y) \,dy \right)^2
- \left(\frac1{|\cdot|}\int_{B_{|\cdot|}} u^2(y) \,dy \right)^2\right\|_2.
\end{align*}
Since $\Hr$ is compactly embedded into $L^q(\RD)$ for all $q>2$,
it follows that $u_n^4 \to u^4$ in $L^2(\RD)$. Moreover as shown in \cite[Lemma 3.2]{byeon},
we also have
\begin{equation*} 
\frac1{|x|}\int_{B_{|x|}} u_n^2(y) \,dy 
\to \frac1{|x|}\int_{B_{|x|}} u^2(y) \,dy 
 \ \ \hbox{in} \ \ L^q(\RD) \ \ \hbox{for} \ \ q>2, \ \ \hbox{as} \ \ n \to \infty,
\end{equation*}
from which we conclude that $E(u_n)\to E(u)$, as $n \to +\infty$.
Analogously one can show that $E'(u_n)[\varphi] \to E'(u)[\varphi]$ for any $\varphi \in H^1_r(\R^2)$
and $E'(u_n)[u_n] \to E'(u)[u]$,
as $n \to \infty$.
\end{proof}

For any $u\in \X$ and $\alpha>0$, we hereafter consider the map
\[
t\in \R^+\longmapsto u_t \in \X,\qquad u_t(x) = t^\a u(t x).
\]
By direct calculations we have $D(u_t)
=t^{6\a-4}D(u)$ and $E(u_t)= t^{8\a-4}E(u)$. 
Thus we get 
\[
I(u_t)=\frac{t^{2\a}}{2} A(u)
+\frac{t^{2\a-2}}{2} \omega B(u)
+t^{4\a} \mu C(u)
+\frac{t^{6\a-4}}{2} q D(u) 
+\frac{t^{8\a-4}}{4} q \mu E(u)
-\frac{t^{(p+1)\a-2}}{p+1} \lambda F(u).
\]
Let
\begin{equation}\label{alpha}
\a \in 
\begin{cases}
(1,+\infty)& \hbox{ if  } p\ge 7,
\\
\left(1,\frac 2{7-p}\right)& \hbox{ if } 5<p< 7.
\end{cases}
\end{equation}
We observe that, fixed $u\in\X\setminus\{0\}$, using \eqref{alpha}, 
the dominant term near $t \sim 0$ of $I(u_t)$ is $t^{2\a -2}$,
which implies that $I(u_t)$ is strictly positive for small $t>0$
and any $u \in \X \setminus \{0\}$.
Furthermore the dominant term near $t \sim +\infty$ among all positive
terms of $I(u_t)$ is $t^{8\alpha-4}$.
Thus under the assumption \eqref{alpha}, 
we see that $8 \alpha -4 < (p+1)\alpha-2$ and hence
$I(u_t) \to -\infty$ as $t \to +\infty$ for any $u \in \X \setminus \{0\}$.
These facts imply that the map
$$\g _u:=t\in (0,+\infty)\longmapsto I(u_t)$$
has a maximum point at a positive level. 
The next lemma shows that this maximum point is the unique critical point.

\begin{lemma}\label{le:max}
Let $u\in \X \setminus\{0\}$. Then the map $\g _u$ 
attains its maximum at exactly one point $t(u)>0$. 
Moreover $\g _u$ is positive and increasing on $(0,t(u))$,
and decreasing for $t>t(u)$.
\end{lemma} 

\begin{proof}
Let $u\in \X \setminus\{0\}$.
A simple computation yields that
\begin{align*}
\g _u'(t)&=\a A(u)t^{2\a-1}
+(\a -1) \omega B(u)t^{2\a-3}
+4\a \mu C(u)t^{4\a-1}
+(3\a -2) q D(u)t^{6\a-5}
\\
&\quad 
+(2\a -1)q \mu  E(u)t^{8\a-5}
-\frac{(p+1)\a-2}{p+1} \lambda F(u)t^{(p+1)\a-3}
\\
&=t^{8\a-5}\left(\frac{\a A(u)}{t^{6\a-4}}
+\frac{(\a -1) \omega B(u)}{t^{6\a-2}}
+\frac{4\a \mu C(u)}{t^{4\a-4}}
+\frac{(3\a -2)q D(u)}{t^{2\a}} 
 \right.
\\
&\qquad \qquad \left.
+(2\a -1) q \mu E(u)
-\frac{(p+1)\a-2}{p+1} \lambda F(u)t^{(p-7)\a+2}\right)
\\
&=:t^{8\a-5}g(t).
\end{align*}
From \eqref{alpha}, it is clear that $\g _u'(t)>0$ for small $t >0$ and $\g _u'(t)<0$ for large $t >0$. 
Then, there exists $t_0 > 0$ such that $\g _u'(t_0) = 0$. 
Moreover, from the choice of $\a $, the function $g(t)$ is strictly decreasing for all $t>0$.
Thus since $\{t > 0 : \g _u'(t) = 0\} = \{t > 0 : g(t) = 0\}$, the critical point of $\g _u(t)$ is unique.
\end{proof}
Let us define 
\begin{align*}
\Gamma(u):=
\g _u'(1)
& =\a A(u) 
+(\a -1)\omega B(u) 
+4\a \mu C(u) 
+(3\a -2) q D(u) \\
&\quad +(2\a -1) q \mu E(u)
-\frac{(p+1)\a-2}{p+1} \lambda F(u)
\end{align*}
and
\begin{equation*}
\M:=\{u \in \X\setminus\{0\} : \Gamma(u)=0 \}.
\end{equation*}
\begin{remark}\label{rem}
From \ef{eq:3.10} and \ef{eq:3.11}, we readily see that
$\Gamma(u)=\alpha N(u)-P(u)$ and hence by Lemma \ref{lem:3.4}, 
any weak solution of \ef{eq:1.1} belongs to $\M$.
\end{remark}

To complete the proof of Theorem \ref{main}, we prepare several lemmas. 
The first one is a direct consequence of Lemma \ref{le:max} since 
\begin{equation}\label{eq:jtf}
\Gamma(u_t)=t\g _u'(t).
\end{equation}

\begin{lemma}\label{le:M}
For any $u\in \X \setminus\{0\}$ there exists a unique $t(u)>0$ such that $u_{t(u)}\in \M$.
\end{lemma}

Next we establish, in the next lemmas, that the functional $I$ is strictly positive on $\M$. 
Indeed, as a first step we have
\begin{lemma}\label{le:pos}
There exists $c>0$ such that for any $u\in \M$
$$
I(u)\ge c\ird [|\n u|^2+u^2+ u^2 |\n u|^2] \,dx.
$$
\end{lemma}

\begin{proof}
Let $u\in \M$. Then we have
\begin{align*}
I(u)&=I(u)-\frac{1}{(p+1)\a-2}\Gamma(u)
\\
&
=\left(\frac 12 -\frac{\a}{(p+1)\a-2}\right)A(u)
+\left(\frac 12 -\frac{\a-1}{(p+1)\a-2}\right) \omega B(u)
+\left(1 -\frac{4\a}{(p+1)\a-2}\right) \mu C(u) \\
&\quad
+\left(\frac 12 -\frac{3\a-2}{(p+1)\a-2}\right) q D(u)
+\left(\frac 14 -\frac{2\a-1}{(p+1)\a-2}\right) q \mu E(u) \\
&=\frac{(p-1)\a -2}{2\big((p+1)\a -2\big)}A(u)
+\frac{(p-1)\a}{2\big((p+1)\a -2\big)} \omega B(u)
+\frac{(p-3)\a -2}{(p+1)\a -2} \mu C(u) \\
&\quad+\frac{(p-5)\a +2}{2\big((p+1)\a -2\big)} q D(u)
+\frac{(p-7)\a +2}{4\big((p+1)\a -2\big)} q \mu E(u).
\end{align*}
By \eqref{alpha}, all coefficients are positive and we conclude.
\end{proof}

\begin{lemma}\label{le:p+1}
There exist $c_1,c_2>0$ such that for any $u\in \M$
\[
\|u\|_{p+1}^{p+1}
\ge
c_1 \ird [|\n u|^2+u^2+ u^2 |\n u|^2] \,dx
\ge c_2.
\] 
\end{lemma}

\begin{proof}
Since $D(u)$ and $E(u)$ are non-negative it follows that
\[
(\a -1) \ird [|\n u|^2 +\omega u^2+ \mu u^2 |\n u|^2] \,dx
-\frac{(p+1)\a -2}{p+1} \lambda \int_{\RD} |u|^{p+1} \,dx
\le \Gamma(u) =0, 
\]
for all $u \in \M$, and we have the first inequality.
\\
Moreover, by the Sobolev inequality, one gets
\[
\ird [|\n u|^2+u^2+u^2 |\n u|^2] \,dx
\le C_1 \| u\|_{p+1}^{p+1}
\le C_2 \| u\|_{H^1}^{p+1}
\le C_2\left( \ird [|\n u|^2+u^2+u^2|\n u|^2] \,dx \right)^{\frac{p+1}{2}}.
\]
This completes the proof.
\end{proof}

Combining Lemmas \ref{le:pos} and \ref{le:p+1}, we have
\begin{lemma}\label{le:posfinal}
There exists $c>0$ such that $I(u)\ge c$, for any $u\in \M$.
\end{lemma}

Let us define
\begin{equation}\label{sigma}
\s :=\inf_{u\in \M}I(u).
\end{equation}
Then by Lemma \ref{le:posfinal}, we infer that $\s >0$.
Moreover by Lemmas \ref{le:max}, \ref{le:M} and from \ef{sigma}, it follows that
\begin{equation} \label{eq:4.3}
\sigma = \inf_{u \in \X \setminus \{0 \} } \max_{t >0} I(u_t).
\end{equation}
Finally we establish the following result.
\begin{proposition} \label{prop:1}
Let $u \in \X$ be a minimizer of $I(u)$ under the constraint $\M$.
Then $u$ is a radial ground state solution of \eqref{eq:1.1}.
Moreover, any radial ground state solution of \ef{eq:1.1} is positive.
\end{proposition}

\begin{proof}
We argue as in \cite[Lemma 2.5]{LWW} or \cite[Theorem 2.2]{RS}.\\
Let $u \in \M$ be a minimizer of the functional $I|_{\M}$.
Then from \ef{eq:4.3}, one has
\begin{equation} \label{eq:4.4}
I(u)= \inf_{v \in \X \setminus \{0 \}} \max_{t >0} I(v_t)
=\inf_{v\in \M} I(v)= \sigma.
\end{equation}
Suppose by contradiction that $u$ is not a weak solution of \ef{eq:1.1}.
Then one can find $\varphi \in C_{0,r}^{\infty}(\R^2)$ such that
\[
I'(u)[\varphi] <-1.
\]
We choose small $\varepsilon>0$ so that
\begin{equation} \label{eq:4.5}
I'(u_t+ \tau \varphi) [\varphi]
\le - \frac{1}{2} \quad \hbox{for} \ |t-1|+|\tau| \le \varepsilon.
\end{equation}
Finally let $\xi\in C_0^{\infty}(\R)$ be a cut-off function satisfying
$0 \le \xi \le 1$, $\xi(t)=1$ for $|t-1| \le \frac{\varepsilon}{2}$
and $\xi(t)=0$ for $|t-1| \ge \varepsilon$.\\
For $t \ge 0$, we construct a path $\eta:\R_+ \to \X$ defined by
\[
\eta(t)=
\begin{cases}
u_t & \hbox{if} \ |t-1| \ge \varepsilon \\
u_t+ \varepsilon \xi(t) \varphi & \hbox{if} \ |t-1| < \varepsilon.
\end{cases}
\]
Then $\eta$ is continuous on the metric space $(\X, d_{\X})$.
Moreover, choosing $\varepsilon$ smaller if necessary, 
it follows that $d_{\X}(\eta(t),0)>0$, for $|t-1| < \varepsilon$.
Next we claim that
\begin{equation} \label{eq:4.6}
\sup_{t \ge 0} I(\eta(t)) < \sigma.
\end{equation}
If $|t-1| \ge \varepsilon$, one has
\[
I(\eta(t)) =I(u_t) <I(u) = \sigma,\]
because the function $t \mapsto I(u_t)$ attains its maximum at $t=1$ for $u \in \M$.\\
If $|t-1| < \varepsilon$, we get
\[
I(u_t+\varepsilon \xi(t) \varphi)
=I(u_t)
+ \int_0^{\varepsilon} I'\big(u_t+\tau \xi(t) \varphi\big) [\xi(t) \varphi] \,d\tau.\]
Then from \ef{eq:4.5}, we obtain
\[
I(\eta(t)) \le I(u_t)-\frac{1}{2} \varepsilon \xi(t) < 
\sigma,\]
yielding that \ef{eq:4.6} holds.\\
Now by \eqref{eq:jtf} and arguing as in Lemma \ref{le:max}, it follows that $\Gamma(\eta(1-\varepsilon))>0$
and $\Gamma(\eta(1+\varepsilon))<0$.
By the continuity of the map $t \mapsto \Gamma(\eta(t))$, 
there exists $t_0 \in (1-\varepsilon,1+\varepsilon)$ such that
$\Gamma(\eta(t_0))=0$.
This implies that $\eta(t_0)=u_{t_0}+\varepsilon \xi(t_0) \varphi \in \M$
and $I(\eta(t_0))< \sigma$ by \ef{eq:4.6}.
This contradicts \ef{eq:4.4},
and hence $u$ is a weak solution of \ef{eq:1.1}. By Remark \ref{rem}, since any weak solution of \ef{eq:1.1} belongs to $\M$, 
we conclude that $u$ is a radial ground state solution.\\
Finally, if $u$ is a minimizer of $I|_{\M}$, then one finds that
$|u|$ is also a minimizer.
Thus we may assume that $u \ge 0$. 
Then, by Proposition \ref{prop:3.2}, we know that $u \in C^2(\R^2)$
and hence we can apply the Harnack inequality \cite{Tr} to conclude that $u>0$.
\end{proof}

\begin{proof}[Proof of Theorem \ref{main}]
Let $\{ u_n \}$ be a minimizing sequence for $I|_\M$, 
namely $\{ u_n \} \subset \M$ and $I(u_n) \to \s $ as $n\to +\infty$. 
By Lemma \ref{le:pos}, the sequences $\{ u_n\}$ and $\{u_n^2 \}$ are bounded in $\Hr$. 
Therefore, there exists $\bar u\in \X$ such that, 
by the compactness result due to \cite{strauss},
up to a subsequence 
\begin{align*}
&u_n\rightharpoonup \bar u  \hbox{ weakly in }\H, \\
&u_n^2\rightharpoonup \bar u^2  \hbox{ weakly in }\H, \\
&u_n\to \bar u\hbox{  in }L^q(\RD) \ \hbox{for any} \ q>2.
\end{align*}
Then, by Lemma \ref{le:p+1}, we infer that $\bar u\neq 0$.\\
Next, by Lemma \ref{le:M}, let us consider $\bar{t}=t(\bar{u})>0$ such that $\bar u_{\bar t}\in \M$. 
Since $[u_n]_{\bar t}\rightharpoonup \bar u_{\bar t}$ and $([u_n]_{\bar t})^2\rightharpoonup \bar u_{\bar t}^2$ weakly in $\H$ as $n \to +\infty$, by Lemmas \ref{le:compD} and \ref{le:compE}, we have
\[
\s \le I(\bar u_{\bar t})
\le \liminf_{n \to \infty} I ([u_n]_{\bar t}).\]
On the other hand, since $\{ u_n \} \subset \M$, 
the function $t \mapsto I ([u_n]_{t})$ reaches its maximum at $t=1$ for all $n \in \N$.
This implies that
\[
\liminf_{n \to \infty} I ([u_n]_{\bar t})
\le \liminf_{n \to \infty} I(u_n)=\s.
\]
Therefore, $\bar u_{\bar t}$ is minimizer of $I$ on $\M$.
Finally by Proposition \ref{prop:1}, 
we conclude that, actually, $\bar u_{\bar t}$ is a radial ground state solution of \eqref{eq:1.1}. 
\end{proof}

\section{Proof of Theorem \ref{thm2}}\label{se:proof2}

In this section, we prove the non-existence result for \ef{eq:1.1} when $1<p<5$.
First we state the following inequality which was obtained in \cite{byeon}.

\begin{proposition}(\cite[Proposition 2.4]{byeon}) \label{prop:5.1} 
For any $u \in H^1_r(\R^2)$, the following inequality holds:
\begin{equation} \label{eq:5.1}
\ird |u|^4 \,dx
\le 2 \| \nabla u\|_2 \left( \ird \frac{h_u^2(|x|)}{|x|^2} u^2 \,dx \right)^{\frac{1}{2}}.
\end{equation}
\end{proposition}

Next we establish the following inequality, 
which cannot be obtained by \ef{eq:5.1} directly.

\begin{proposition} \label{pr:5.2}
For any $u \in \X$, the following inequality holds:
\begin{equation} \label{eq:5.2}
\ird |u|^6 \,dx 
\le 4 \left( \ird u^2 |\nabla u|^2 \,dx \right)^{\frac{1}{2}}
\left( \ird \frac{h_u^2(|x|)}{|x|^2} u^4 \,dx \right)^{\frac{1}{2}}.
\end{equation}
\end{proposition}

\begin{proof}
The proof is same as that of Proposition \ref{prop:5.1}.
By the density, we may assume that $u \in C_0^{\infty}(\R^2)$.
Then by the Fubini Theorem and the Schwarz inequality, one has
\begin{align*}
\ird |u|^6 \,dx 
&= 2\pi \int_0^{\infty} ru^2(r) \left( 
\int_r^{\infty} - \big( u^4(s) \big)' \,ds \right) \,dr \\
&\le 8\pi \int_0^{\infty} \int_0^{\infty} ru^2(r)
|u(s)|^3 |u'(s)| \chi_{\{s>r\}} \,ds \,dr \\
&= 8 \pi \int_0^{\infty} |u(s)|^3 |u'(s)|
\left( \int_0^s ru^2(r) \,dr \right) \,ds \\
&= 4 \ird \frac{|u|^3|\nabla u|}{|x|} h_u(|x|) \,dx \\
&\le 4 \left( \ird u^2 |\nabla u|^2 \,dx \right)^{\frac{1}{2}}
\left( \ird \frac{h_u^2(|x|)}{|x|^2} u^4 \,dx \right)^{\frac{1}{2}}.
\end{align*}
\end{proof}

\begin{proof}[Proof of Theorem \ref{thm2}]
Suppose that $1<p<5$ and let $u \in \X$ be a solution of \ef{eq:1.1}.
We distinguish three cases: $0<q<\frac{1}{3}$, $\frac{1}{3} \le q <2$
and $q \ge 2$.\\
First we consider the case $0<q<\frac{1}{3}$.
From \ef{eq:5.1}, \ef{eq:5.2}
and by the Young inequality, it follows that
\begin{align} \label{eq:5.4}
\ird |u|^4 \,dx &\le \| \nabla u\|_2^2 + \ird \frac{h_u^2(|x|)}{|x|^2} u^2 \,dx, \nonumber \\
\ird |u|^6 \,dx &\le 2 \ird u^2 |\nabla u|^2 \,dx
+2 \ird \frac{h_u^2(|x|)}{|x|^2} u^4 \,dx. 
\end{align}
Then since $N(u)=0$, we obtain
\begin{align*}
0 &\ge (1-3q) \| \nabla u\|_2^2 +2 \mu (2-q) \ird u^2 |\nabla u|^2 \,dx 
+\ird \Big( \omega u^2 +3q u^4 +q\mu u^6 -\lambda |u|^{p+1} \Big) \,dx \\
&\ge \ird \Big( \omega u^2 +3q u^4 +q\mu u^6 -\lambda |u|^{p+1} \Big) \,dx.
\end{align*}
In the case $\frac{1}{3} \le q<2$, 
we slightly modify the use of \ef{eq:5.1} to obtain
\[
\ird |u|^4 \,dx \le \frac{1}{3q} \| \nabla u\|_2^2
+3q \ird \frac{h_u^2(|x|)}{|x|^2} u^2 \,dx.\]
From this estimate, the Nehari identity and \ef{eq:5.4}, one gets
\[
0 \ge \left(1- \frac{1}{3q} \right) \| \nabla u \|_2^2
+2\mu(2-q) \ird u^2 |\nabla u|^2 \,dx
+\ird \Big( \omega u^2 +u^4 +q\mu u^6 - \lambda |u|^{p+1} \Big) \,dx.\]
Finally in the case $q \ge 2$, we apply the following estimate which is derived from 
\ef{eq:5.2}:
\[
\ird |u|^6 \,dx \le \frac{4}{q} \ird u^2 |\nabla u|^2 \,dx
+q \ird \frac{h_u^2(|x|)}{|x|^2} u^4 \,dx.\]
Then we have
\[
0 \ge \left(1-\frac{1}{3q} \right) \| \nabla u\|_2^2
+4\mu \left( 1- \frac{2}{q} \right) \ird u^2 |\nabla u|^2 \,dx
+\ird \Big( \omega u^2 +u^4 +2 \mu u^6 - \lambda |u|^{p+1} \Big) \,dx.\]
Now we define $g: \R \to \R$ where
\[
g(t):=
\begin{cases}
\omega t^2 +3q t^4 +q\mu t^6 -\lambda |t|^{p+1} & \hbox{ if } 0<q<\frac{1}{3}, \\
\omega t^2 + t^4 +q\mu t^6 -\lambda |t|^{p+1} & \hbox{ if }  \frac{1}{3}\le q <2 , \\
\omega t^2 + t^4 +2\mu t^6 -\lambda |t|^{p+1} & \hbox{ if }  q \ge 2 . 
\end{cases}
\]
Then one has 
\begin{equation} \label{eq:5.5}
\ird g(u) \,dx\le 0.
\end{equation}
We observe that for given $q$, $\mu$ and $\lambda>0$, 
there exists $\bar{\omega}>0$ such that 
$g(t)>0$ for $\omega \ge \bar{\omega}$ and $t \ne 0$.
Indeed for $0<q < \frac{1}{3}$, one has
\[
g'(t)=t \Big( 2 \omega +12 q t^2 +6q \mu t^4 
-(p+1) \lambda |t|^{p-1} \Big).\]
Since $1<p<5$, we can apply the Young inequality to
\[
(p+1)\lambda |t|^{p-1} = \left( \frac{24 q \mu |t|^4}{p-1} 
\right)^{\frac{p-1}{4}} \cdot (p+1) \lambda \left( \frac{p-1}{24 q \mu} \right)^{\frac{p-1}{4}}\]
and obtain
\[
(p+1) \lambda |t|^{p-1}
\le 6q\mu t^4 + \frac{5-p}{4} \left( (p+1)\lambda \right)^{\frac{4}{5-p}}
\left( \frac{p-1}{24 q \mu} \right)^{\frac{p-1}{5-p}}.\]
Thus it follows that
\begin{equation*} 
g'(t) \ge t\left( 2 \omega -
\frac{5-p}{4} \left( (p+1)\lambda \right)^{\frac{4}{5-p}}
\left( \frac{p-1}{24 q \mu} \right)^{\frac{p-1}{5-p}}
+12 q t^2 \right) \ \hbox{for} \ t >0.
\end{equation*}
Taking $\omega$ larger, we have $g'(t)>0$ for $t>0$.
Other cases can be treated in a same way.
Similarly one has $g'(t)<0$ for $t<0$ and hence
$g(t)>0$ for $t \ne 0$, as claimed.
This and \ef{eq:5.5} imply that $u \equiv 0$
and hence the proof is complete.
\end{proof}

\begin{remark} \label{rem:5.3}
		It is easy to check that $\bar{\omega}$, defined in Theorem \ref{thm2}, increases as $q$ decreases.
		In other words, if we fix $\omega$, we have to take $q$ smaller 
		in order to obtain nontrivial solutions of \ef{eq:1.1}.
		This is exactly the situation studied in \cite{cunha} for the case $\mu=0$.\\
		Analogously $\bar{\omega}$ increases as $\mu$ decreases.
		In particular, when $1<p<3$, we notice that $\bar{\omega}$ can be chosen 
		independent of $\mu$ and this is consistent with the result obtained in \cite{AD}.
		On the other hand, if $3 \le p<5$, 
		we have to choose larger $\bar{\omega}$ as $\mu$ becomes smaller.
		We expect, therefore, that, for fixed $\omega$ and $3 \le p<5$, 
		we are able to find a nontrivial solution of \ef{eq:1.1}
		provided that $\mu$ is sufficiently large.
\end{remark}

\medskip
\subsection*{Acknowledgment}
The first two authors are partially supported by  a grant of the group GNAMPA of INdAM and FRA2016 of Politecnico di Bari. 
The third author is supported by JSPS Grant-in-Aid for Scientific Research (C) (No. 15K04970).

\end{document}